\documentclass[12pt]{amsart}

\usepackage{amsmath,amssymb,amsthm}
\usepackage{amsfonts,dsfont,mathrsfs,eucal}
\usepackage{graphicx}
\usepackage{color}
\usepackage{verbatim} 
\usepackage[margin=30mm]{geometry}
\newcommand{\rene}{\color{black}}
\newcommand{\normal}{\color{black}}

\newtheorem{theorem}{Theorem}[section]
\newtheorem{lemma}[theorem]{Lemma}

\newtheorem{corollary}[theorem]{Corollary}

\theoremstyle{definition}

\theoremstyle{remark}
\newtheorem{remark}[theorem]{Remark}

\numberwithin{equation}{section}

\newcommand{\Ee}{\mathds{E}} 
\newcommand{\Pp}{\mathds{P}} 
\newcommand{\dom}{\mathrm{D}} 
\newcommand{\gen}{A} 
\newcommand{\caputo}{\partial} 
\newcommand{\BV}{\mathcal{M}} 
\newcommand{\LL}{\mathcal{L}} 

\newcommand{\I}{\mathds{1}}

\def\R{\mathds{R}}
\def\N{\mathds{N}}

\def\a{\alpha}

\renewcommand{\leq}{\leqslant}
\renewcommand{\geq}{\geqslant}
\renewcommand{\le}{\leqslant}
\renewcommand{\ge}{\geqslant}
\renewcommand{\Re}{\ensuremath{\operatorname{Re}}}

\begin{document}

\title{Reflected spectrally negative stable processes and their governing equations}

\author{Boris Baeumer}
\address{Boris Baeumer, Department of Mathematics and Statistics, University of Otago,
Dunedin, NZ} \email{bbaeumer@maths.otago.ac.nz}
\author{Mih\'aly Kov\'acs}
\address{Mih\'aly Kov\'acs, Department of Mathematics and Statistics, University of Otago,
Dunedin, NZ} \email{mkovacs@maths.otago.ac.nz}
\author{Mark M.\ Meerschaert}
\address{Mark M.\ Meerschaert, Department of Probability and Statistics,
Michigan State University}
\email{mcubed@stt.msu.edu}
\urladdr{http://www.stt.msu.edu/users/mcubed/}
\thanks{M.M.\ Meerschaert was partially
supported by NSF grants DMS-1025486 and DMS-0803360, and NIH grant R01-EB012079.}

\author{Ren\'e L.\ Schilling}
\address{Ren\'e L. Schilling, Institut f\"ur Mathematische Stochastik, Technische Universit\"at Dresden,
Germany}
\email{rene.schilling@tu-dresden.de}
\urladdr{http://www.math.tu-dresden.de/sto/schilling/}

\author{Peter Straka}
\address{Peter Straka, School of Mathematics and
Statistics, The University of New South Wales, Australia}
\email{p.straka@unsw.edu.au}

\begin{abstract}
This paper explicitly computes the transition densities of a spectrally negative stable process with index greater than one, reflected at its infimum.  First we derive the forward equation using the theory of sun-dual semigroups.  The resulting forward equation is a boundary value problem on the positive half-line that involves a negative Riemann-Liouville fractional derivative in space, and a fractional reflecting boundary condition at the origin.  Then we apply numerical methods to explicitly compute the transition density of this space-inhomogeneous Markov process, for any starting point, to any desired degree of accuracy.  Finally, we discuss an application to fractional Cauchy problems, which involve a positive Caputo fractional derivative in time.
\end{abstract}



\maketitle
\section{Introduction}

Consider a spectrally negative (no positive jumps) stable L\'evy process $Y_t$ with characteristic function
\begin{equation}\label{YtCF}
\Ee[e^{ik Y_t}]=e^{t(ik)^\alpha}
\end{equation}
for some $1<\alpha\leq 2$.  If $\alpha=2$, then $Y_t$ is a Brownian motion with variance $2t$.  Now define
\begin{equation}\label{ZtDef}
Z_t = Y_t-\inf\{Y_s: 0 \le s \le t\} .
\end{equation}
The reflected stable process \eqref{ZtDef} is also the recurrent extension of the process $Y_t$ killed at zero, which instantaneously and continuously
leaves zero, see Patie and Simon \cite{PatieSimon}.  Let $C_\infty(\R)$ denote the Banach space of continuous functions $f:\R\to\R$ that tend to zero as $|x|\to\infty$, with the supremum norm.  We say that a time-homogeneous Markov process $X_t$ is a \emph{Feller process} if the semigroup $T_t f(x)=\Ee[f(X_{t+s})|X_s=x]$ satisfies $T_t f\in C_\infty(\R)$ and $T_t f\to f$ as $t\to 0$ in the Banach space (supremum) norm, for all $f\in C_\infty(\R)$.  It is not hard to show (see Theorem \ref{bwgenerator} in this paper) that $Z_t$ is a Feller process, and since $Z_t\geq 0$ by definition, the space-inhomogeneous Markov process $Z_t$ lives on the state space $[0,\infty)$.

If $\alpha=2$, then this process is called ``reflected Brownian motion'' and the governing differential equation (forward Kolmogorov equation) for the transition density $p(x,y,t)$ of $Z_{t+s}=y$ given $Z_s=x$ is the diffusion equation $\partial_t p(x,y,t)=\partial^2_y p(x,y,t)$ together with the reflecting boundary condition
\begin{equation}\label{rBmBC}
\partial_y p(x,y,t)\bigg\vert_{y=0+}:=\lim_{h\to 0+}\frac{p(x,y+h,t)-p(x,y,t)}{h}\bigg\vert_{y=0}=0 \quad\text{for all $t>0$,}
\end{equation}
(i.e., the normal derivative vanishes) see for example It\^{o} and McKean \cite[Eq.\ 8]{ItoMcKean}.  This paper extends to the case of a reflected stable process.  The stable process $Y_t$ without reflection is a space-homogeneous Markov process, so its transition density $p(y,t)$ is independent of the initial state $x$.  This density solves a fractional diffusion equation, $\partial_t p(y,t)=D^\alpha_{-y} p(y,t)$ that involves a negative Riemann-Liouville fractional derivative in space, see \eqref{RLneg} below for the definition.  The fractional derivative reduces to the usual second derivative in the case $\alpha=2$.  The forward equation for the reflected stable process turns out to be the fractional diffusion equation $\partial_t p(x,y,t)=D^\alpha_{-y} p(x,y,t)$ with the fractional reflecting boundary condition
\begin{equation}\label{PosGrunwald}
 D_{-y}^{\alpha-1} p(x,y,t)\bigg\vert_{y=0+}:=   \lim_{h\to 0+}\frac{1}{h^{\alpha-1}}\sum_{k=0}^\infty w^{\alpha-1}_k p(x,y+kh,t)\bigg\vert_{y=0} =0 \quad\text{for all $t>0$,}
\end{equation}
using the (fractional) binomial coefficients
\[ w^\alpha_k := (-1)^k\binom\alpha k .\]
When $\alpha=2$ we have $w^{\alpha-1}_0=1$, $w^{\alpha-1}_1=-1$, and $w^{\alpha-1}_k=0$ for $k>1$, so that \eqref{PosGrunwald} reduces to the classical condition \eqref{rBmBC}, i.e., the one-sided first derivative.  In either case ($\alpha=2$ or $1<\alpha<2$), the boundary term enforces a no-flux condition at the point $y=0$ in the state space.

The connection between probability and differential equations has profound consequences for mathematics \cite{allouba1,barlow,coupleCTRW,MNV}, and for its applications in science and engineering \cite{GorenfloSurvey,Metzler2000,MetzlerKlafter,scalas1}, including a probabilistic method called particle tracking for solving fractional differential equations, by exploiting the associated Markov process  \cite{MRMT,ParticleTracking,timeLangevin}.   More details on fractional calculus, and its connection to probability theory, may be found in the recent book of Meerschaert and Sikorskii \cite{FCbook}.  Since fractional derivatives are nonlocal operators, the appropriate specification of boundary conditions requires a new approach \cite{agrawal2002solution,DiegoDCN06,kilbas2006north,luchko2010some,metzler2000boundary,voller2010exact}.  For example, one can apply the theory of Volterra integral equations \cite{pruss} or general nonlocal operators \cite{lehoucq}. The results in this paper can help clarify the meaning of reflecting boundary conditions for fractional diffusion.

We believe that this idea will find many useful applications, both inside and outside mathematics.
As a first application, we show in Theorem \ref{MPsubordTh} that a reflected stable process can be used to solve a fractional Cauchy problem, in which the usual first time derivative is replaced by a Caputo fractional derivative of order $0<\beta<1$.

\bigskip\noindent
\textbf{Notation.} We write $C_\infty[0,\infty)$ for the Banach space of continuous functions that vanish at infinity, i.e., $\lim_{x\to\infty} u(x) = 0$, with the uniform norm $\|u\| = \sup_{x\in [0,\infty)} |u(x)|$. Its topological dual is the space of (signed) Radon measures $\BV_b[0,\infty)$, and by $\BV_{ac}[0,\infty)$ we mean the absolutely continuous (with respect to Lebesgue measure) elements in $\BV_b[0,\infty)$. On $\BV_b[0,\infty)$ we use vague (weak-$*$) convergence; i.e.,\ $\mu_n \to \mu$ if, and only if, $\int u \,d\mu_n \to \int u\,d\mu$ for all $u\in C_\infty[0,\infty)$. The subscripts $c, b, ac, \infty$ stand for `compact support', `bounded', `absolutely continuous' and `vanishing at infinity'.
Fractional integrals and derivatives in the Riemann-Liouville sense are denoted by $I^\alpha$ and $D^\alpha$, see \eqref{RLintpos}--\eqref{RLneg} while Caputo derivatives are written as $\caputo^\alpha$, see \eqref{Caputo}. \rene Finally, $\LL  g(s) = \LL [g(t)](s) = \int_0^\infty e^{-st}\,g(t)\,dt$
denotes the usual Laplace transform, and $\LL_{-\infty}  g(s) = \LL_{-\infty} [g(t)](s) = \int_{-\infty}^\infty e^{-st}\,g(t)\,dt$
is the bilateral Laplace transform. \normal

\section{The reflected stable process}
Given a real number $\alpha>0$ that is not an integer, define the positive Riemann-Liouville fractional integral
\begin{equation}\label{RLintpos}
I_{x}^\alpha f(x)=\frac {1}{\Gamma(\alpha)} \int_0^\infty f(x- y) y^{\alpha-1}\,dy=\frac {1}{\Gamma(\alpha)} \int_{-\infty}^x f(y) (x-y)^{\alpha-1}\,dy,
\end{equation}
the negative Riemann-Liouville fractional integral
\begin{equation}\label{RLintneg}
I_{-x}^\alpha f(x)=\frac {1}{\Gamma(\alpha)} \int_0^\infty f(x+ y) y^{\alpha-1}\,dy=\frac {1}{\Gamma(\alpha)} \int_x^\infty f(y) (y-x)^{\alpha-1}\,dy,
\end{equation}
the positive Riemann-Liouville fractional derivative
\begin{equation}\label{RLpos}
D^\alpha_{x} f(x):=\frac {d^n}{d x^n}I_{x}^{n-\alpha} f(x)=\frac 1{\Gamma(n-\alpha)} \frac {d^n}{d x^n}\int_{-\infty}^x f(y) (x-y)^{n-\alpha-1}\,dy,
\end{equation}
and the negative Riemann-Liouville fractional derivative
\begin{equation}\label{RLneg}
D^\alpha_{-x} f(x):=\frac {d^n}{d (-x)^n}I_{-x}^{n-\alpha} f(x)=\frac {(-1)^n}{\Gamma(n-\alpha)} \frac {d^n}{d x^n}\int_x^\infty f(y) (y-x)^{n-\alpha-1}\,dy,
\end{equation}
where $n-1< \alpha <n$.  If $\alpha\in(1,2)$, then $n=2$.  See \cite{Bajlekova,FCbook,Samko} for more details.

Let $Z_t$ be the stochastic process defined in \eqref{ZtDef}, where $Y_t$ is a stable L\'evy process with index $1<\alpha< 2$ and characteristic function \eqref{YtCF}.  Next we will show that $Z_t$ is a conservative time-homogeneous Markov process whose (backward) semigroup $T_tf(x)=\Ee[f(Z_{t+s})|Z_s=x]$ is strongly continuous ($\|T_tf-f\|\to 0$ as $t\downarrow 0$), contractive ($\|T_tf\|\leq \|f\|$), and analytic
(the mapping $t\mapsto T(t)f$ has an analytic
extension to the sectorial region $\{re^{i\theta}\in{\mathbb
C}: r>0, |\theta|<\alpha\}$ for some $\alpha>0$)
on the Banach space $X=C_{\infty}[0,\infty)$, and give a core for the generator. Recall that a core $C_\gen$ of a closed linear operator $\gen$ is a subset of its domain $\dom(\gen)$ that is dense within the domain in the graph norm; i.e., for each $f\in \dom(\gen)$ there exists a sequence $\{f_n\}\subset C_\gen$ such that $f_n\to f$ and $\gen f_n\to \gen f$.

Write
\begin{equation}\label{Sb}
\begin{aligned}
    S_b
    := \Big\{f\in C_\infty&[0,\infty)\,:\, f''\in C(0,\infty),\; f''(x)=O(1) \text{\ as\ } x\to\infty,\\
    &f''(x) = O(x^{\alpha-2}) \text{\ as\ } x\to 0,\; f'\in C_b(0,\infty), \; f'(0+)=0 \Big\}
\end{aligned}
\end{equation}
and denote by $\partial^\alpha_x$ the (positive) Caputo fractional derivative of order $\alpha>0$, which can be defined by 
\begin{equation}\label{Caputo}
    \caputo^\alpha_{x} f(x)
    =I^{n-\alpha}_x f^{(n)}(x)
    =\frac 1{\Gamma(n-\alpha)}\int_0^x (x-y)^{n-1-\alpha}f^{(n)}(y) \,dy
\end{equation}
where $n-1<\alpha<n$ and $f^{(n)}$ is the $n$th derivative of $f$.  The Caputo fractional derivative differs from the Riemann-Liouville form \eqref{RLpos} because the operations of differentiation  and (fractional) integration do not commute in general.  For example, when $0<\alpha<1$ we have
\begin{equation}\label{RLtoCaputoto}
\caputo^\alpha_{x} f(x)=D^\alpha_{x} f(x)-f(0)\frac{x^{-\alpha}}{\Gamma(1-\alpha)}
\end{equation}
for suitably nice bounded functions (e.g., see \cite[p.\ 39]{FCbook}).

\begin{theorem}\label{bwgenerator}
Let $Z_t$ denote the reflected process \eqref{ZtDef} where $Y_t$ is a stable L\'evy process with index $\alpha = 1/\beta\in(1,2)$ and characteristic function \eqref{YtCF}.  Then $Z_t$ is a Feller process, the transition semigroup $T_tf(x)=\Ee[f(Z_{t+s})|Z_s=x]$ on $C_\infty[0,\infty)$ is analytic, with generator $\gen f=\partial^\alpha_{x}f$ for all $f\in C_\gen=\{f\in S_b:\partial^\alpha_{x}f\in C_\infty[0,\infty)\}$, where $\partial^\alpha_{x}$ is the Caputo fractional derivative \eqref{Caputo}, and $C_\gen$ is a core of $\gen$.
\end{theorem}

\begin{proof}
Define the running infimum $I_t = \inf\{Y_s: 0 \le s \le t\}$ and the running supremum $S_t = \sup\{Y_s: 0 \le s \le t\}$.  Let $\hat Y_t=-Y_t$ denote the dual process, and let $\hat I_t$ and $\hat S_t$ denote the running infimum and supremum of $\hat Y_t$, respectively.  Since $\hat Y_t$ is also a L\'evy process, it follows from \cite[Section VI.1, Proposition 1]{Bertoin} that $\hat S_t-\hat Y_t$ is a Feller process, and since $\hat S_t=-I_t$, it follows that $\hat S_t-\hat Y_t=-I_t+Y_t=Z_t$.  Hence $Z_t$ is a Feller process.

It follows from \cite[Proposition 4]{Bernyk2011} that $C_\gen\subset \dom(\gen)$ and $\gen f=\partial_x^\alpha f$ for all $f\in C_\gen$. Note that the extension from $f''$ bounded to $|f''(x)|=O( x^{\alpha-2})$ at $x=0+$ is also mentioned in the proof.

Since $\gen$ generates a strongly continuous contraction semigroup, the resolvent operators $R(\lambda,\gen):=(\lambda-\gen)^{-1}$ exist for all $\Re\lambda>0$ and they are bounded operators. We will now show that
\begin{equation}\label{Mittag-Leffler}
    R(\lambda,\gen)g(x)= -\alpha x^{\alpha-1}E'_\alpha(\lambda x^\alpha)\star g(x)+\frac{\mathcal L g(\lambda^{1/\alpha})}{\lambda^{1-1/\alpha}}E_\alpha(\lambda x^\alpha)
\end{equation}
for all $g\in C_\infty[0,\infty)$, where $\mathcal L g$ is the Laplace transform, $\star$ is the convolution operator, and the Mittag-Leffler function
$E_\alpha(x)=\sum_{n=0}^\infty {x^n}/{\Gamma(1+\alpha n)}$.

Let $R_{\lambda,g}$ denote the right-hand side of \eqref{Mittag-Leffler}. Since
$$
    \LL [E_\alpha(\lambda x^\alpha)](s)
    = \frac{s^{\alpha-1}}{s^\alpha-\lambda}
    \quad\text{and}\quad
    \LL [\alpha x^{\alpha-1}E'_\alpha(\lambda x^\alpha)](s) = \frac 1{s^\alpha-\lambda},
$$
see e.g.\ \cite{Haubold,MainardiWaves}, it follows that
\begin{equation}\label{resL}
    \LL{R_{\lambda,g}}(s)
    =\frac{\LL g(s)}{\lambda-s^\alpha} - \frac{s^{\alpha-1}}{\lambda-s^\alpha} \frac{\lambda^{1/\alpha}\LL g(\lambda^{1/\alpha})}{\lambda}
\end{equation}
Using the fact that $[f\star g]' (x) =[f'\star g](x)+f(0)g(x)$ it follows that $R_{\lambda,g}$ is twice differentiable for any $g\in C_\infty[0,\infty)\cap C^2[0,\infty)$.   Equation \eqref{Caputo} implies
that for any $f\in C^2[0,\infty)$ we have
$$
    \LL[\partial_x^\alpha f](s)
    =s^{\alpha-2}\left(s^2\LL f(s)-s f(0)-f'(0)\right)
    =s^\alpha \LL f(s)-s^{\alpha-1}f(0).
$$\rene
Taking the Laplace transform of $\lambda R_{\lambda,g}-\partial_x^\alpha R_{\lambda,g}$,
we therefore obtain that
\begin{equation}\label{resLexp1}
  \begin{split}
    \LL\big[\lambda R_{\lambda,g}- \partial_x^\alpha R_{\lambda,g}\big](s)
    &= \lambda\LL R_{\lambda,g}(s)-s^\alpha \LL R_{\lambda,g}(s)+s^{\alpha-1}R_{\lambda,g}(0)\\
    &= \LL g(s)-\frac{s^{\alpha-1}}{\lambda^{1-1/\alpha}}\LL g(\lambda^{1/\alpha})+s^{\alpha-1}R_{\lambda,g}(0)\\
    &= \LL g(s)
  \end{split}
\end{equation}
for any $g\in C_\infty[0,\infty)\cap C^2[0,\infty)$, and hence
$R(\lambda,\gen)g=R_{\lambda,g}$.
By continuous extension, \eqref{Mittag-Leffler} holds for all $g\in C_\infty[0,\infty)$.

To show that $C_\gen$ is a core, note that for $g\in C^2_\infty[0,\infty):=C^2[0,\infty)\cap C_\infty[0,\infty)$,
$$\partial^\alpha_x R(\lambda,\gen)g=\lambda R(\lambda,\gen)g-g\in C_\infty[0,\infty)$$
 and hence $R(\lambda,\gen)C^2_\infty[0,\infty)\subset C_\gen$. Pick $f\in \dom(\gen)$ and $g=\lambda f-\gen f$. Since $C^2_\infty[0,\infty)$ is dense in $C_\infty[0,\infty)$, there exists a sequence $\{g_n\}\subset C^2_\infty[0,\infty)$ with $g_n\to g$. Thus, $f_n = R(\lambda,\gen)g_n\to f$ and $\gen f_n=\lambda f_n-g_n\to\lambda f-g=\gen f$. Since $f_n\in R(\lambda,\gen) C^2_\infty[0,\infty) \subset C_A$, we see that $C_A$ is a core.

Finally we show that $\{T_t\}_{t\geq 0}$ is an analytic semigroup.  Since $R(\lambda,\gen)$ is a bounded operator for all $\Re\lambda>0$, a general result from the theory of semigroups \cite[Corollary 3.7.12]{Arendt2001} shows that $\{T_t\}$ is analytic if for some some $M>0$ we have
 \begin{equation}\label{resEst}
 \|\lambda R(\lambda,\gen)g\|\le M\|g\|
 \end{equation}
for all $\Re \lambda>0$ and all $g\in C_\infty[0,\infty)$. Then the result follows from Lemma \ref{LemResolvent}, and this completes the proof.
\end{proof}
\normal

\begin{remark}
Patie and Simon \cite{PatieSimon} show that the reflected stable process $Z_t$ in Theorem~\ref{bwgenerator} has the backward generator
\begin{equation}\label{Ldef}
\gen f(x)=f'(0)\frac{x^{1-\alpha}}{\Gamma(2-\alpha)}+\int_0^x f''(x-y) \frac{y^{1-\alpha}}{\Gamma(2-\alpha)}\,dy .
\end{equation}
They also give the exact domain of the generator \cite[Proposition 2.2]{PatieSimon}.  If $f\in S_b$, then $f'(0)=0$, and $\gen f$ reduces to the Caputo fractional derivative \eqref{Caputo}.
\end{remark}


In view of Theorem~\ref{bwgenerator}, $T_tf(x)=\Ee[f(Z_{t+s})|Z_s=x]$ is a strongly continuous, analytic semigroup on the Banach space $X:=C_\infty[0,\infty)$ with the supremum norm, with generator $\gen f(x)=\caputo_x^\alpha f(x)$ for $f\in S_b$ such that $\caputo^\alpha_{x}f\in X$.  The dual (or adjoint) semigroup $T^*_t$ is defined on the dual space $X^*= \BV_b[0,\infty)$ of finite signed Radon measures on $[0,\infty)$ equipped with the total variation norm:  Given a measure $\mu\in \BV_b[0,\infty)$, use the Jordan decomposition to write $\mu=\mu^{+}-\mu^{-}$ uniquely as a difference of two positive measures, and define $\|\mu\|=\mu^{+}[0,\infty)+\mu^{-}[0,\infty)$.
The dual semigroup satisfies
\begin{equation}\label{dualdef}
\int T_tf(x)\mu(dx)=\int f(x)[T^*_t\mu](dx)
\end{equation}
for all $f\in C_\infty[0,\infty)$ and all $\mu\in \BV_b[0,\infty)$.  See \cite[Section 2.5]{Engel2000} for more details.  In probabilistic terms, since $T_tf(x)=\Ee[f(Z_{t+s})|Z_s=x]$ for this time-homogeneous Markov process, equation \eqref{dualdef} implies that
\begin{equation*}\begin{split}
\int T_tf(x)\mu(dx)&=\int \Ee[f(Z_{t+s})|Z_s=x]\mu(dx)\\
&=\int f(y)P_t(dy,\mu)=\int f(y)[T^*_t\mu](dy)\\
 \end{split}\end{equation*}
where $P_t(y,\mu)=\int P(y,x,t)\mu(dx)$ and $P(y,x,t)=\Pp[Z_{t+s}\leq y|Z_s=x]$ is the transition probability distribution of the Markov process $Z_t$.  Hence, if $\mu$ is the probability distribution of $Z_s$, then $T^*_t\mu(dy)=P_t(dy,\mu)$ is the probability distribution of $Z_{t+s}$.  The dual semigroup is also called the \emph{forward semigroup} associated with the Markov process $Z$, since it maps the probability distribution forward in time.

Next we will compute the generator $\gen^*$ of the forward semigroup.  This is the 
adjoint of the generator $\gen$ of the backward semigroup, in the sense that
\[\int \gen f(x)\mu(dx)=\int f(x)[\gen^*\mu](dx)\]
for all $f\in \dom(\gen)$ and $\mu\in \dom(\gen^*)$.  Theorem~\ref{fwdgen} will show that every measure $\mu\in \dom(\gen^*)$ has a Lebesgue density $g\in L^1[0,\infty)$, so that $\mu(dy)=g(y)\,dy$, and that the adjoint $\gen^*g:=\gen^*\mu$ of the positive fractional Caputo derivative $\gen f(x)=\caputo^\alpha_x f(x)$ in our setting is the negative Riemann-Liouville fractional derivative $\gen^*g(y)=D_{-y}^\alpha g(y)$ using \eqref{RLneg}.


The forward semigroup $T_t^*$ of a Markov process is not, in general, strongly continuous on $\BV_b[0,\infty)$.  That is, there exist measures $\mu$ such that $T_t^*\mu\not\to \mu$ in the total variation norm as $t\downarrow  0$. For example, if $T^*_t$ is the forward semigroup associated with the diffusion equation $\partial_t p=\partial_x^2 p$, and $\mu=\delta_0$ is a point mass at the origin, then $T_t^*\mu$ is a Gaussian probability measure with mean 0 and variance $2t$ for all $t>0$, and since $\mu\{0\}=1$ and $T_t\mu\{0\}=0$ for all $t>0$, we have $\|T_t\mu-\mu\|=1$ for all $t>0$ in the total variation norm.

To handle this situation, we define the \emph{sun dual space} of $X:= C_\infty[0,\infty)$ as
$$
    X^\odot:=\{\mu\in X^*:\lim_{t\downarrow 0}\|T^*_t\mu-\mu\|= 0\} ,
$$
a closed subspace of $X^*=\BV_b[0,\infty)$ on which the forward semigroup is strongly continuous.  It follows from basic semigroup theory \cite[Section 2.6]{Engel2000} that for $\mu\in X^\odot$, $T_t^*\mu\in X^\odot$ for all $t\ge 0$, and $X^\odot=\overline{\dom(\gen^*)}$. The restriction of $\{T_t^*\}_{t\ge 0}$ to $X^\odot$ is called the \emph{sun dual semigroup} $\{T_t^\odot\}_{t\ge 0}$ with generator $\gen^\odot \mu=\gen^*\mu$ for all $\mu\in \dom(\gen^\odot)$, where
\begin{equation}\label{dualdomain}
\dom(\gen^\odot)=\{\mu\in \dom(\gen^*):\gen^*\mu\in X^\odot\} .
\end{equation}
For the reflected stable process, we will show in Theorem~\ref{fwdgen} that $C^\odot_\infty[0,\infty)$ is the space of absolutely continuous elements of $\BV_b[0,\infty)$,
$$
    \BV_{ac}[0,\infty)
    = \big\{\mu\in\BV_b[0,\infty) \,:\, \mu(dy)=g(y)\,dy\text{ for some }g\in L^1[0,\infty)\big\}.
$$
and we will derive the forward equation of the reflected stable process on the sun-dual space. For a general bounded measure $\mu\in \BV_b[0,\infty)$, we will then prove in Corollary~\ref{weaksoln} that $T_t^*\mu$ can be computed as the vague limit of $T_t^\odot\mu_n$, where $\mu_n\to \mu$ vaguely, and $\mu_n\in C^\odot_\infty[0,\infty)$ for all $n$.

%

%

\begin{theorem}\label{fwdgen}
 Let $Z_t$ denote the Feller process \eqref{ZtDef}, where $Y_t$ is a stable L\'evy process with index $\alpha = 1/\beta\in(1,2)$ and characteristic function \eqref{YtCF}, with (backward) semigroup $T_tf(x)=\Ee[f(Z_{t+s})|Z_s=x]$ on $C_\infty[0,\infty)$.   Then
 $
    C^\odot_\infty[0,\infty)
    =
    \BV_{ac}[0,\infty)
$
 and the generator $\gen^\odot g:=\gen^\odot \mu$ of the sun-dual semigroup $\{T^\odot(t)\}_{t\ge0}$ is given by
\begin{equation}\label{sundualgen}
    \gen^\odot g(y)=D^\alpha_{-y}g(y)
\end{equation}
with domain $\dom(\gen^\odot)=\{g\in L^1[0,\infty)\,:\,D^\alpha_{-y}g(y)\in L^1[0,\infty),D^{\alpha-1}_{-y}g(0)=0\}$.
\end{theorem}

\begin{proof}
Suppose $\gen^* \mu=\nu\in \BV_b[0,\infty)$ for some $\mu\in\BV_b[0,\infty)$, so that $\int \gen f(x) \mu(dx)=\int f(x)\nu(dx)$ for all $f\in \dom(\gen)$. Set $v(x):=\nu[0,x]$ for $x\ge 0$ and $v(x)=0$ for $x<0$. If $f\in S_b$ with $\caputo^\alpha_x f\in C_{\infty}[0,\infty)$, then it follows from Theorem~\ref{bwgenerator} that $f\in \dom(\gen)$ and $\gen f=\caputo^\alpha_{x}f$.  It is obvious from the Definition \eqref{Caputo} that $\caputo_x^{\alpha-1}f'(x)=\caputo^\alpha_{x}f(x)$.  Let
\[I_x^\alpha f(x)=\int_0^x\frac{(x-y)^{\alpha-1}}{\Gamma(\alpha)}f(y)\,dy \]
denote the positive Riemann-Liouville fractional integral \eqref{RLintpos} of order $\alpha>0$ for a function $f\in C_{\infty}[0,\infty)$, and apply the general formula \cite[Eq.\ (1.21)]{Bajlekova} $I^{\alpha-1}_x\partial_x^{\alpha-1}f'(x)=f'(x)-f'(0)$
to see that $f'(x)=I^{\alpha-1}_x\partial_x^{\alpha-1}f'(x)=I^{\alpha-1}_x\partial_x^{\alpha}f(x)=I^{\alpha-1}_x \gen f(x)$.  Since $f(x)\to 0$ as $x\to\infty$, and $v(0)=0$ for $x<0$, we can apply the integration by parts formula \cite[Theorem 19.3.13]{Hildebrandt}
\[\int_a^b f(x)\nu(dx)=f(b)v(b)-f(a)v(a)-\int_a^b v(x)f'(x)dx\]
with $a<0$, and then let $b\to\infty$, to see that
\[  \int_0^\infty f(x)\nu(dx)=-\int_0^\infty f'(x)v(x)\,dx .\]
Thus, for all $f\in S_b$ with $\partial^\alpha_x f\in C_{\infty}[0,\infty)$, a Fubini argument yields
 \begin{equation}\label{maincalc}
 \begin{split}
   \int_0^\infty f(x)\nu(dx)
    =&-\int_0^\infty \int_0^x\frac{(x-y)^{\alpha-2}}{\Gamma(\alpha-1)}\gen f(y)\,dy\, v(x)\,dx\\
    =&-\int_0^\infty \int_y^\infty\frac{(x-y)^{\alpha-2}}{\Gamma(\alpha-1)} v(x)\,dx\, \gen f(y)\,dy\\
    =&\int_0^\infty  \gen f(y)\,\mu(dy).
 \end{split}\end{equation}

Next we will show that $S:=\{\gen f:f\in S_b\text{ with }\partial^\alpha_x f\in C_\infty[0,\infty)\}$ is dense in $C_\infty[0,\infty)$, and then it will follow that
any measure $\mu\in \dom(\gen^*)$ has a Lebesgue density
\begin{equation}\label{g}
    g(y)=-\int_y^\infty\frac{(x-y)^{\alpha-2}}{\Gamma(\alpha-1)} v(x)\,dx=-I_{-y}^{\alpha-1}v(y)
\end{equation}
where $\gen^* \mu=\nu$ and $v(x)=\nu[0,x]$.  Let $C_c^\infty[0,\infty)$ denote the space of smooth functions with compact support, i.e., such that $h(x)=0$ for all $x>M$, for some $M>0$. It is not hard to check that the space $Q=\{h\in C_c^\infty[0,\infty):\int h=0\}$ is dense in $C_\infty[0,\infty)$. Then we certainly have
$$
    \lim_{x\to \infty} I_x^\alpha h(x)=\lim_{x\to \infty}\int_0^{M} \frac{(x-s)^{\alpha-1}-x^{\alpha-1}}{\Gamma(\alpha)}h(s)\,ds=0,
$$
and therefore, the function $f(x)=I_x^\alpha h(x)$ is an element of $C_{\infty}[0,\infty)$ for any $h\in Q$.  Elementary estimates suffice to check that $f\in S_b$ as well. Since the positive Caputo derivative is a left inverse of the positive Riemann-Liouville integral \cite[Eq.\ (1.21)]{Bajlekova}, we also have  $\partial_x^\alpha f(x)=\partial_x^\alpha I_x^\alpha h(x)=h(x)\in C_{\infty}[0,\infty)$.  Hence $h=\gen f\in S$ for all $h\in Q$, and thus $Q\subseteq S$.  Now for any $f \in C_\infty[0,\infty)$ there exists a sequence $f_n\to f$ in the supremum norm, with $f_n\in S$ for all $n$.  Then some simple estimates can be used to verify that $\int f(y)\mu(dy)=\int f(y)g(y)\,dy$, and it follows that the measure $\mu$ has the Lebesgue density $g$.    Hence we can identify $\dom(\gen^*)$ with a subspace of $L^1[0,\infty)$.

 %
Next we will show that this subspace is dense in $L^1[0,\infty)$.  Define
$$
    \phi_n(x):=\tfrac{1}{\Gamma(\alpha)} \left(\tfrac 1n -x\right)^{\alpha-1}\,\I_{[0,1/n)}(x) ,
$$
and note that $D^{\alpha-1}_{-x}\phi_n(x)\equiv 1$ for $x\in(0,1/n)$ by a straightforward computation.  Note that the set
$U:=\{g(x)-[D^{\alpha-1}_{-x}g(0)]\phi_n:g\in C_c^\infty[0,\infty), n \in \mathbb N\}$ is dense in $L^1[0,\infty)$, and that $D^\alpha_{-x}g(x)\in L^1[0,\infty)$ for all $g\in U$. 
Furthermore, $U$ is a subset of
$$
    G= \big\{ g\in L^1[0,\infty)\,:\,D^{\alpha}_{-x}g\in L^1[0,\infty), \; D_{-x}^{\alpha-1}g(0)=0\big\} .
$$
For any $g\in G$ and $f\in S_b$ with $\partial^\alpha_x f\in C_{\infty}[0,\infty)$, we have $D^{\alpha-1}_{-x}g(0)=0$ and $f'(0)=0$.   Then Theorem~\ref{bwgenerator}, a Fubini argument, and integration by parts (twice) using equation \eqref{RLneg} yields
   \begin{equation}\label{MMM1}
     \begin{split}
       \int_0^\infty \gen f(y) \,g(y)\,dy =& \int_0^\infty \int_0^y \frac{(y-x)^{1-\alpha}}{\Gamma(2-\alpha)}f''(x)\,dx\, g(y)\,dy\\
       =& \int_0^\infty\int_x^\infty \frac{(y-x)^{1-\alpha}}{\Gamma(2-\alpha)} g(y)\,dy\,f''(x)\,dx\\
 =& \int_0^\infty I^{2-\alpha}_{-x} g(x)f''(x)\,dx\\
       =& \int_0^\infty\,D_{-x}^{\alpha-1} g(x)f'(x)\,dx\\
       =& \int_0^\infty f(x)\,D_{-x}^{\alpha} g(x)\,dx.
     \end{split}
   \end{equation}
Furthermore, as $\{f\in S_b:\gen f\in C_\infty[0,\infty)\}$ is a core, for all $f\in \dom(\gen)$ there exists a sequence $f_n$ in the core such that $f_n\to f$ and $\gen f_n\to \gen f$. Hence \eqref{MMM1} holds for all $f\in \dom(\gen)$ and therefore
 $g\in \dom(\gen^*)$ and $\gen^*g=D_{-x}^\alpha g$ for any $g\in G$.  Since $U$ is dense in $L^1[0,\infty)$, and $U\subseteq G\subseteq \dom(\gen^*)$, it follows that $\dom(\gen^*)$ is dense in $L^1[0,\infty)$.  Since $C^\odot_\infty[0,\infty)$ is the smallest closed set containing $\dom(\gen^*)$ by definition, and since $L^1[0,\infty)$ is a closed subspace of $\BV_b[0,\infty)$, we have shown that $L^1[0,\infty)=C^\odot_\infty[0,\infty)$.

Equation \eqref{dualdomain} implies that $\nu=\gen^*\mu$ is an element of $C^\odot_\infty[0,\infty)$ for any $\mu\in \dom(\gen^\odot)$, and therefore, we have $\nu(dx)=h(x)\,dx$ for some $h\in L^1[0,\infty)$, as well as $\mu(dy)=g(y)\,dy$ for some $g\in L^1[0,\infty)$.  Since $v(x)=\nu[0,x]$, it follows that $h(x)=v'(x)$.  Since $D_{-y}^{\alpha-1}$ is a left inverse of $I_{-y}^{\alpha-1}$ in general, it follows from  \eqref{g} that $v(y)=-D_{-y}^{\alpha-1}g(y)$.  Then we have
\begin{equation}\label{dualgencalc}
h(x)=v'(x)=\frac d{dx}\left[ -D^{\alpha-1}_{-x}g(x)\right]=D^{\alpha}_{-x}g(x)=\gen^* g(x)
\end{equation}
for all $g\in \dom(\gen^\odot)$, which proves the generator formula \eqref{sundualgen}. Equation \eqref{dualgencalc} also shows that $D^{\alpha}_{-x}g(x)\in  L^1[0,\infty)$ for all $g\in \dom(\gen^\odot)$, and since $v$ is continuous with $v(0)=0$, it follows that $D_{-x}^{\alpha-1}g(0)=v(0)=0$ for all $g\in \dom(\gen^\odot)$.  This proves that $\dom(\gen^\odot)\subseteq G$.  Since $\dom(\gen^\odot)$ is defined as the set of $g\in \dom(\gen^*)$ such that $\gen^*g\in L^1[0,\infty)$, it follows from \eqref{MMM1} that $g\in \dom(\gen^\odot)$ for all $g\in G$, so that $G\subseteq \dom(\gen^\odot)$ as well, which completes the proof.
\end{proof}

Theorem~\ref{fwdgen} establishes the forward equation of the reflected stable process $Z_t$, for certain initial conditions.  It shows that, for any initial condition $\mu_0(dx)=g(x)dx$ where $g\in L^1[0,\infty)$, $\mu_t:=T^\odot \mu$ solves the Cauchy problem
\[\partial_t \mu(t)=\gen^\odot \mu(t);\quad \mu(0)=\mu_0\]
on the sun-dual space. This implies that, for any probability density $p_0(x)$ such that $p_0(x)=0$ for $x<0$, the function $p(x,t)=T^\odot p_0(x)$ solves the forward equation
\begin{equation}\label{FBC2}
\partial_t p(x,t)=D_{-x}^\alpha p(x,t);\quad p(x,0)=p_0(x), \quad D_{-x}^{\alpha-1} p(x,t)\Big|_{x=0} \equiv 0 .
\end{equation}


The next two results will allow us to compute the transition probability density $y\mapsto p(x,y,t)$ of the time-homogeneous Markov process $y=Z_{t+s}$ for any initial state $x=Z_s$, by applying Theorem~\ref{fwdgen} to a sequence of initial conditions $\mu_n\in X^\odot$ such that $\mu_n\to \delta_x$. The first result shows that the transition density exists.

\begin{corollary}\label{CorTDexists}
The semigroup $T^{\odot}_t$ is a strongly continuous bounded analytic semigroup on $L^1(\R_+)$  and
the transition probability distributions $T_t^*\delta_x$ have smooth densities $y\mapsto p(x,y,t)$ for all $t>0$ and all $x\ge 0$.
\end{corollary}

\begin{proof}
It is well-known that the spectra of $\gen$ and $\gen^*$ coincide, and $R(\lambda,\gen^*)=R(\lambda, \gen)^*$
for all $\lambda$ in the resolvent set of $\gen$ (and $\gen^*$). Therefore, since $\gen$ is a sectorial operator, being the generator of a bounded analytic semigroup, it follows that $\gen^*$ is a sectorial operator as well, and hence $\gen^*$ generates a bounded analytic semigroup (not necessarily strongly continuous at $0$) on $C^*_\infty[0,\infty) = \BV_b[0,\infty)$ by \cite[Theorem 3.7.1]{Arendt2001} which coincides with $T_t^*$ for all $t>0$ by the uniqueness of the Laplace transform. Therefore, the restriction of $\gen^*$ to $\overline{\dom(\gen^*)}=L^1[0,\infty)$ (i.e., the operator $\gen^{\odot}$) generates a strongly continuous bounded analytic semigroup $T_t^{\odot}$ on $L^1[0,\infty)$ by \cite[Remark 3.7.13]{Arendt2001}. Furthermore, \cite[Remark 3.7.20]{Arendt2001} shows that for all $n\in \N$ and $t>0$ we have that $T^*_t\mu\in \dom((\gen^*)^n)$ for all $\mu\in\BV_b[0,\infty)$ and that $T^{\odot}_tf\in \dom((\gen^{\odot})^n)$ for all $f\in L^1[0,\infty)$.
Since $\dom(\gen^*)\subset L^1[0,\infty)$, it follows that $T^*_s\mu\in L^1[0,\infty)$ for all $s>0$ and $\mu\in\BV_b[0,\infty)$. Therefore for all $t>0$ and $\mu\in \BV_b[0,\infty)$ we have
$$
    T^*_t\mu=T^*_{\frac{t}{2}}T^*_{\frac{t}{2}}\mu=T^{\odot}_{\frac{t}{2}}T^*_{\frac{t}{2}}\mu \in \dom((\gen^{\odot})^n)
$$
 for all $n\in \N$. Thus by taking $\mu=\delta_x$ it follows that $T^*_t\delta_x\in \dom((\gen^\odot)^n)$ for all $n$. Using Theorem~\ref{fwdgen}, we have that $\dom(\gen^\odot)=\{g\in L^1[0,\infty):D^\alpha_{-y}g(y)\in L^1[0,\infty),D^{\alpha-1}_{-y}g(0)=0\}$, and then it follows that $\left(D^\alpha_{-y}\right)^n T^*_t\delta_x\in L^1[0,\infty)$ for all $n$.  In particular, the transition probability distribution $T_t^*\delta_x$ has a density function $y\mapsto p(x,y,t)$ for all $t>0$ and all $x\ge 0$.  To see that this function is smooth note that any positive integer $m$ can be written in the form $m=n\alpha-\beta$ for some integer $n\geq 1$ and some positive real number $\beta<\alpha$.  A straightforward calculation shows that for $f\in \dom((\gen^\odot)^n)$ we have $\left(D^\alpha_{-y}\right)^n f=D^{n\alpha}_{-y}f$ and $I_{-y}^\beta D^{n\alpha}_{-y}f=D^{n\alpha-\beta}_{-y}f\in L^1[0,\infty)$ for all $\beta<\alpha$ and $n\ge 1$.
Since $D^m_{-y}f=(-1)^m (d/dy)^m f$, we have
 $(d/dy)^m p(x,y,t)\in L^1[0,\infty)$ for all $m$.
\end{proof}

\begin{corollary}\label{weaksoln}
    Let $\{\mu_n\}\subset C_\infty^\odot[0,\infty) = \BV_{ac}[0,\infty)$ such that $\mu_n\to \mu$ vaguely as $n\to\infty$ for some $\mu\in \BV_b[0,\infty)$, then $T_t^\odot \mu_n\to T_t^*\mu$ vaguely as $n\to\infty$.
\end{corollary}

\begin{proof}
Since $T_t^\odot \mu_n=T_t^*\mu_n$ for all $n\geq 1$, for all $\phi\in C_\infty[0,\infty)$ we have
  \begin{equation}\begin{split}\int \phi(x) [T_t^\odot \mu_n](dx)=&\int \phi(x) [T_t^*\mu_n](dx)=\int [T_t\phi](x) \mu_n(x)dx\\
  &\xrightarrow{n\to\infty}\int [T_t\phi](x)\mu(dx) =\int \phi(x) [T_t^* \mu](dx),\end{split}\end{equation}
and the result follows.
\end{proof}

In Section~\ref{sec5}, we will apply Corollary~\ref{weaksoln} to compute the transition densities of the reflected stable process to any desired degree of accuracy, by solving the forward equation numerically.  For any initial state $Z_s=x$, we will approximate the initial condition $\mu_0=\delta_x$ in the numerical method by a sequence of measures $\mu_n$ with $L^1$-densities, and then Corollary~\ref{weaksoln} guarantees that the resulting solutions converge to the transition density in the supremum norm as $n\to\infty$.

\begin{remark}
Using integration by parts, one can write the backward generator in the form of an integro-differential operator (e.g., see Jacob \cite{Jacob})
\begin{equation}\label{Lpseudo}
\gen f(x)= b(x)f'(x)+\int \left[f(x+y)-f(x)-yf'(x)\right] \phi(x,dy) ,
\end{equation}
with coefficients
\begin{equation}\begin{split}\label{characteristics}
b(x)&=\frac {x^{1-\alpha}}{\Gamma(2-\alpha)}\quad\text{and}\\
\phi(x,dy)&=\frac {\alpha(\alpha-1)}{\Gamma(2-\alpha)}|y|^{-1-\alpha}dy \I_{(-x,0)}(y) +\frac {\alpha-1}{\Gamma(2-\alpha)}x^{-\alpha}\varepsilon_{-x}(dy) .
 \end{split}\end{equation}
The jump intensity $\phi(x,dy)$ describes the behavior of the process $Z_t$, which truncates jumps of the stable process $Y_t$ starting at the point $x>0$ in the state space, so that a jump (they are all negative) of size $|y|>x$ is changed to a jump of size $x$.  This keeps the sample paths of $Z_t$ inside the half-line $[0,\infty)$.  Since the drift $b(x)$ is unbounded, the existence of a Markov process $Z_t$ with generator \eqref{Lpseudo} would not follow from general theory (e.g., see \cite[Section 4.5]{Ethier} or \cite[Section 3]{Schilling98}). Hence the reflected stable process $Z_t$ is an interesting example of a Markov process with unbounded drift coefficient.
\end{remark}

\section{Transition density of the reflected stable process}\label{sec5}

To the best of our knowledge, there is no known analytical formula for the transition density $y\mapsto p(x;t,y)$ of the reflected stable process $y=Z_{t+s}$ started at $x=Z_s>0$.  In this section, we compute and plot this transition density, by numerically solving the associated forward equation \eqref{FBC2}.  The existence and smoothness of these transition densities is guaranteed by Corollary~\ref{CorTDexists}.  Corollary~\ref{weaksoln} shows that $T^\odot_t g_n(y)dy\to T^*_t\delta_x(dy)$ in the supremum norm as $n\to\infty$ for any sequence of functions $g_n\in L^1[0,\infty)$ such that $g_n(y)dy\to \delta_x(y)dy$ (vague convergence).  We take $g_n(y)=n\ \I_{[x,x+1/n]}(y)$.  Then the solutions $T^\odot_t g_n$ provide estimates of the transition density to any desired degree of accuracy.  Theorem~\ref{fwdgen} shows that $T^\odot_t g_n$ solves the forward equation \eqref{FBC2}.  Hence, we can compute the transition densities of the stable process by solving the forward equation numerically, with this initial condition.

In order to compute the probability density $p(x,y,t)$ numerically, we consider the fractional boundary value problem
\begin{equation}
  \label{transDE}
  \partial_t u(y,t)=D_{-y}^\alpha u(y,t);\quad u(y,0)=\delta_x(y);\quad D^{\alpha-1}_{-y}u(0,t)=0 .
\end{equation}
We develop forward-stepping numerical solutions $u_h(y_i,t)$ that estimate $u(y_i,t)$ at locations $y_i=ih$ for $i=0,1,\ldots,N$ over an interval $[0,y_{max}]$ in the state space, where $y_{max}$ is chosen large enough and $h$ is chosen small enough so that enlarging the domain further, or making the step size smaller, has no appreciable effect on the computed solutions (e.g., the resulting graph does not visibly change).  We approximate the delta function initial condition $u(y,0)=\delta_x(y)$ by setting $u_h(y_i,0)=1/h$ for $y_i=x$ and $u_h(y_i,0)=0$ otherwise, a numerical representation of the initial condition $g_n(y)=n\ \I_{[x,x+1/n]}(y)$ with $h=1/n$.

Numerical methods for fractional differential equations are an active area of research.  One important finding \cite{frade,2sided} is that a shifted version of the Gr\"unwald finite difference formula \eqref{PosGrunwald} for the fractional derivative is required to obtain a stable, convergent method.  Hence we approximate
\begin{equation}\label{MMMxxx}
    D_{-y}^\alpha u(y_i,t)\approx\frac 1{h^\alpha}\sum_{k=0}^{N+1-i} w^\alpha_k  u_h(y_{i+k-1},t)
    \quad\text{where}\quad
    w^\alpha_k := (-1)^k\binom\alpha k
\end{equation}
for $i\geq 2$. Note that this approximation of $D_{-y}^\alpha u(y_i,t)$ does not depend on the value of $u_h$ at the boundary $y_0=0$.   We enforce the boundary condition at $y_0=0$ at each step; i.e., we set
$$u_h(y_0,t)=-\sum_{k=1}^{N}w_k^{\alpha-1}u_h(y_k,t)$$
so that
\[D_{-y}^{\alpha-1} u(y_0,t)\approx \sum_{k=0}^{N}w_k^{\alpha-1}u_h(y_k,t)=0\]
since $w_0^{\alpha-1}=1$.
Finally, for $i=1$ we approximate
\begin{equation}\begin{split}
D^\alpha_{-y} u(y_1,t)\approx&\frac1{h^\alpha}\sum_{k=0}^{N}w^\alpha_k u_h(y_{1+k-1},t)\\
=& \frac1{h^\alpha}\left(\sum_{k=1}^{N}w_k^{\alpha}u_h(y_{k},t)-\sum_{k=1}^{N}w_k^{\alpha-1}u_h(y_k,t)\right)\\
=&-\frac1{h^\alpha}\sum_{k=1}^{N}w_{k-1}^{\alpha-1}u_h(y_k,t)
\end{split}
\end{equation}
using an elementary identity for fractional binomial coefficients $w_k^\alpha-w_k^{\alpha-1}=-w_{k-1}^{\alpha-1}$.

This leads to the following linear system of ordinary differential equations,
\begin{equation}\label{numericPDE}
  \frac{d}{dt} \begin{pmatrix}
    u_h(y_1,t)\\
    \vdots\\
    \vdots\\
    \vdots\\
    u_h(y_N,t)
  \end{pmatrix}
  = \frac 1{h^\alpha}\begin{pmatrix}
    -1 & \alpha-1 & \ldots & \ldots & -w_{N-1}^{\alpha-1}\\
    1 & -\alpha & w_2^\alpha & \ldots & w_{N-1}^{\alpha} \\
    0 & 1 & \ddots & \ddots & w_{N-2}^\alpha\\
    \vdots & \ddots & \ddots& \ddots & \vdots\\
    0 & \ldots & 0 & 1 & -\alpha
  \end{pmatrix}
   \begin{pmatrix}
    u_h(y_1,t)\\
    \vdots\\
    \vdots\\
    \vdots\\
    u_h(y_N,t)
  \end{pmatrix}
\end{equation}
whose solution can be approximated using any numerical ODE solver (indeed, this is a linear system of the form $u'=Au$, so it has a solution $u(t)=e^{tA}u(0)$ for any initial condition).
The resulting numerical solutions with $y_{\max}=12$, $h=0.01$, starting points $x=0,1,2,4$, and $\alpha =1.2,1.8$ are depicted in Figures~\ref{fig1_2} and~\ref{fig1_8}, where each frame represents a snapshot at times $t=0.5,1$ and $2$ respectively.  The $L^1$-error of the numerical solutions for $x=0$
decays linearly with $h$. For $h=0.01$ the $L^1$-error is less than $0.05$ for $\alpha=1.2$, and less than $0.004$ for $\alpha=1.8$, for every case plotted. A short MATLAB code to compute the numerical solution is included in the Appendix.

\begin{figure}
 \includegraphics[width=16cm]{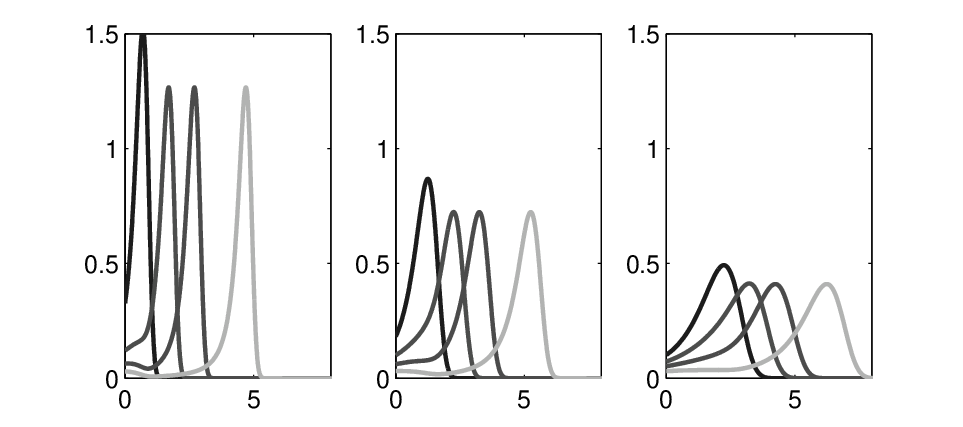}\\
  \caption{The transition densities of $p(x,y,t)$ with $\alpha=1.2$, $x=0,1,2,4$ (left to right) at times $t=0.5$ (left panel), $t=1$ (middle panel), and $t=2$ (right panel). }\label{fig1_2}

\end{figure}

\begin{figure}
  \includegraphics[width=16cm]{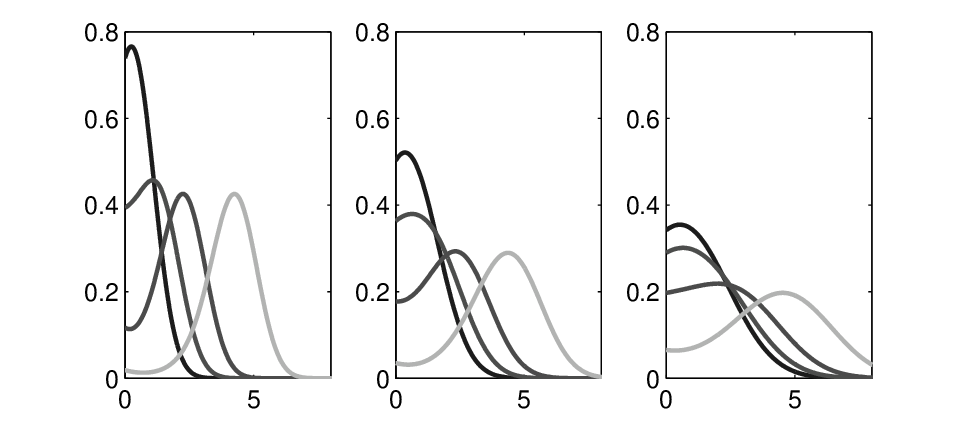}\\
  \caption{The transition densities of $p(x,y,t)$ with $\alpha=1.8$, $x=0,1,2,4$ (left to right) at times $t=0.5$ (left panel), $t=1$ (middle panel), and $t=2$ (right panel). }\label{fig1_8}
\end{figure}

\begin{remark}
It is interesting to note that the matrix in \eqref{numericPDE} is essentially the rate matrix of a discrete state Markov process in continuous time.  Extending the state space to $N=\infty$, we obtain a Markov process  $Z^h_t$ on the state space $\{ih:i>0\}$ that approximates the reflected stable process, with $u_h(x_i,t)=\Pp(Z^h_t=ih)$.  The transition rate from state $ih$ to state $jh$ for $i>j>1$ is $w^\alpha_{i-j+1}\approx \alpha(\alpha-1) (i-j)^{-\alpha-1}/\Gamma(2-\alpha)$, the jump intensity of the stable process $Y_t$, in view of \cite[Eq.\ (2.5)]{FCbook}.   The transition rate from state $ih$ for $i>1$ to state $jh$ for $j=1$, in the first row of the rate matrix, is $w^{\alpha-1}_{i-1}\approx i^{-\alpha}/\Gamma(1-\alpha)$, the rate at which the process $Y_t$ would jump into the negative half-line.   This can be computed as $\phi(-\infty,-ih)$ where $\phi(dy)=\alpha(\alpha-1)|y|^{-1-\alpha}dy/\Gamma(2-\alpha)$ is the L\'evy measure of the process $Y_t$, e.g., see \cite[Proposition 3.12]{FCbook}.
\end{remark}

\begin{remark}
As noted in the introduction, the fractional boundary condition in \eqref{transDE} is a natural extension of the boundary condition \eqref{rBmBC} for Brownian motion on the half-line.  The fractional boundary condition in \eqref{transDE} can be written in Gr\"unwald finite difference form using \cite[Proposition 2.1]{FCbook} to arrive at \eqref{PosGrunwald}.
\end{remark}

\begin{remark}
Bernyk, Dalang and Peskir \cite[Appendix]{Bernyk2011} computed the backward generator of a general reflected stable L\'evy process.  Caballero and Chaumont \cite[Theorem 3]{CC} compute the backward generator of a killed stable L\'evy process.  It may be possible to develop the forward equation and compute the transition density for those process, using the methods of this paper.  This would be interesting for applications to fractional diffusion, since it could elucidate the relevant fractional boundary conditions.
\end{remark}

\section{Fractional Cauchy problems}
In this section, we show that the reflected stable process \eqref{ZtDef} with index $1<\alpha\leq 2$ can be used as a time change to solve the fractional Cauchy problem
\begin{equation}\label{fCp}
\partial_t^\beta p(x,t)=L p(x,t);\quad p(x,0)=f(x)
\end{equation}
 of order $\beta=1/\alpha$, when $L$ generates a Feller process.  The time-fractional Caputo derivative in \eqref{fCp} is defined by
\begin{equation}\label{CaputoDvt}
    \caputo^\beta_t f(t) =\frac{1}{\Gamma(1-\beta)}\int_{0}^\infty f'(t-r) r^{-\beta}dr ,
\end{equation}
a special case of \eqref{Caputo}.  Fractional Cauchy problems are useful in a wide variety of practical applications \cite{GorenfloSurvey,Metzler2000,MetzlerKlafter,scalas1}, and the next result allows a Markovian particle tracking solution for such problems \cite{MRMT,ParticleTracking,timeLangevin}.

\begin{theorem}\label{MPsubordTh}
For any $\beta\in[1/2,1)$, let $Z_t$ be given by \eqref{ZtDef}, where $Y_t$ is an $\alpha$-stable L\'evy process with characteristic function \eqref{YtCF} for $\alpha = 1/\beta$. If $X_t$ is an independent Markov process such that $T_t f(x)=\Ee^x[f(X_t)]$ forms a uniformly bounded, strongly continuous semigroup with generator $L$ on some Banach space ${\mathbb B}$ of real valued functions, then $p(x,t)=\Ee^x[f(X_{Z_t})]$ solves the fractional Cauchy problem \eqref{fCp} for any $f\in\dom(L)$, the domain of the generator.
\end{theorem}

\begin{proof}
Theorem 3.1 in \cite{fracCauchy} states that if $u(x,t)$ solves the Cauchy problem $\partial_t u(x,t)=Lu(x,t)$; $u(x,0)=f(x)$ on ${\mathbb B}$ for some $f\in{\rm Dom}(L)$, then the solution to the associated fractional Cauchy problem \eqref{fCp} on ${\mathbb B}$ is given by
\begin{equation}\label{fCp-soln}
    p(x,t)=\int_0^\infty u(x,r)h(r,t)\,dr
\end{equation}
where
\begin{equation}\label{Etdens}
    h(r,t)=\frac{t}{\beta}r^{-1-1/\beta}\,g_\beta(tr^{-1/\beta})
\end{equation}
and $g_\beta(t)$ is the stable probability density function with Laplace transform
\[\LL g(s)=\int_0^\infty e^{-st} g_\beta (t)\,dt=e^{-s^\beta}\]
for some $0<\beta<1$.  Since $Y_t$ has no positive jumps, the first passage time $D_t:=\inf\{r>0:Y_r>t\}$ is a stable subordinator with $\Ee[e^{-sD_t}]=\exp(-ts^{1/\alpha})$, and the supremum process $S_t=\sup\{Y_r:0\leq r\leq t\}$ is also the first passage time $E_t=\inf\{u>0:D_u>t\}$ of $D_t$, see Bingham \cite{Bingham73}.   Note that $\Pp(D_{E_t}>t)=1$ \cite[Theorem III.4]{Bertoin}. Apply \cite[Corollary 3.1]{limitCTRW} (or Exercises 29.7 and 29.18 in Sato \cite{Sato}) to see that \eqref{Etdens} is also the probability density of $E_t$.   It follows from \cite[Section VI.1, Prop. 3]{Bertoin} that
\begin{equation}\label{EqFluct}
\Pp(S_t \ge x)= \Pp(Z_t \ge x)\quad\text{for all $t>0$ and all $x>0$.}
\end{equation}
Hence \eqref{Etdens} is also the probability density of $Z_t$, and the theorem follows.
\end{proof}

\begin{remark}
An alternate proof of Theorem~\ref{MPsubordTh} uses the reflection principle for spectrally negative stable L\'evy processes.  Extending the usual argument for reflected Brownian motion, let $\tau_x=\inf\{u>0:Y_u> x\}$.  Since $Y_t$ is self-similar, we have $\Pp(Y_t>0)=1/\alpha$ for every $t>0$ (e.g., see \cite[Theorem 4.1 (i)]{duality}). Defining $Y^x_t:=Y_{t-\tau_x}-x$ for $t\geq \tau_x$, we have
\begin{equation*}\begin{split}
\Pp(S_t\geq x)&=\Pp(S_t\geq x,Y_t> x)+\Pp(S_t\geq x,Y_t\leq x)\\
&=\Pp(Y_t> x)+\Pp(\tau_x\leq t,Y^x_t\leq 0)
\end{split}\end{equation*}
and since $\Pp(\tau_x\leq t,Y^x_t\leq 0)=(1-\alpha^{-1})\Pp(S_t\geq x)$ it follows that
    \begin{equation}\label{Andre2}
\Pp(S_t \ge x)=\alpha \Pp(Y_t > x)=\Pp(Y_t > x | Y_t \ge 0) \quad\text{for all $t>0$ and all $x>0$.}
\end{equation}
An application of Zolotarev duality for stable densities \cite[Theorem 4.1 (ii)]{duality} implies
\begin{equation}\label{dualityEq}
\Pp(E_t > x) = \Pp(Y_t > x | Y_t \ge 0) \quad\text{for all $t>0$ and all $x>0$.}
\end{equation}
Then the theorem follows using \eqref{EqFluct}, \cite[Theorem 3.1]{fracCauchy} and \cite[Corollary 3.1]{limitCTRW}.
\end{remark}

\begin{remark}
Theorem~\ref{MPsubordTh} confirms a conjecture in the paper \cite[Remark 5.2]{duality}.  There we set $Z_t=Y_{\sigma(t)}$ where $\sigma(t)=\inf\{u>0:H_u>t\}$ and $H_u=\int_0^u \I_{Y_s>0}\,ds$.  In essence, the negative excursions are cut away, and the positive excursions are joined together without any gaps in time.  Since $Y_t$ has no positive jumps, any up-crossing at the origin is a renewal point, so this process has the same distribution as \eqref{ZtDef}.
\end{remark}

\section{Resolvent estimates}
In this section, we develop bounds on the norm of the resolvent used in Theorem \ref{bwgenerator}.  Note that both terms in the formula \eqref{Mittag-Leffler} for the resolvent $R(\lambda,A)$  diverge as $x\to\infty$.  For example, it follows directly from \cite[Eq. (6.4)]{Haubold} that $E_\alpha(\lambda x^\alpha)\sim \alpha^{-1} e^{\lambda^{1/\alpha}x}$ as $x\to\infty$.  Hence it is useful to begin by establishing an alternative representation.  Recall that $Y_t$ is a negatively skewed stable process with index $1<\alpha\leq 2$ and characteristic function \eqref{YtCF}.  Let $g_\alpha(x)$ denote the probability density function of $-Y_1$, a totally positively skewed stable law taking values on the entire real line.

\begin{lemma}\label{LemResolvent1}
For any $g\in C_\infty[0,\infty)$ and any $\Re\lambda>0$ we have
\begin{equation}\begin{split}\label{resStable}
  R(\lambda,A)g(x)=&\int_0^\infty\int_0^\infty e^{-\lambda t}\frac{1}{t^{1/\alpha}}g_\alpha\left(\frac{x-\xi}{t^{1/\alpha}}\right) g(\xi)\,dt\,d\xi\\
  &+\frac{\mathcal L g(\lambda^{1/\alpha})}{\lambda^{1-1/\alpha}}\int_0^\infty e^{-\lambda t}\frac{x}{\alpha t^{1+1/\alpha}}g_\alpha\left(\frac{x}{t^{1/\alpha}}\right)\,dt .
  \end{split}
\end{equation}
where $\mathcal L g$ is the usual Laplace transform.
\end{lemma}

\begin{proof}

First we will show that the right-hand side of \eqref{resStable} vanishes for any $x<0$.  Recall that $g_\alpha(x)$ is a standard positively skewed stable density with index $1<\alpha\leq 2$ and characteristic function $\exp((-ik)^\alpha)$, and hence $g_\alpha(-x)$ is a standard negatively skewed stable density with index $1<\alpha\leq 2$ and characteristic function $\exp((ik)^\alpha)$, i.e., the density of $Y_1$ in  \eqref{YtCF}.  For $0<\beta<1$, let $g_\beta(x)$ denote the standard positively skewed stable density with characteristic function $\exp(-(-ik)^\beta)$.  Since this density is supported on the positive real line, we can also write
\begin{equation}\label{dualityLem1gLT}
\mathcal L g_\beta(s)=\exp(-s^\beta)\quad\text{for all $\Re s>0$,}
\end{equation}
(e.g., see \cite[Lemma 2.2.1]{zolotarev-61}).  Apply the Zolotarev Duality Theorem for stable densities \cite[Theorem 2.3.1]{zolotarev-61} to see that when $\beta=1/\alpha$ we have
\begin{equation}\label{dualityWeird}
g_\alpha(-x)=x^{-1-\alpha}g_{\beta}(x^{-\alpha})\quad\text{for all $x>0$}.
\end{equation}
Then for $x<0$, the first term in \eqref{resStable} is
 \begin{equation*}\begin{split}
    \int_{-\infty}^x\int_0^\infty & e^{-\lambda t}\frac{1}{t^{1/\alpha}} g_\alpha\left(\frac{u}{t^{1/\alpha}}\right) g(x-u)\,dt\,du \\
    =& \int_{-\infty}^x \int_0^\infty  e^{-\lambda t}\frac{1}{t^{1/\alpha}} \left(-\frac{u}{t^{1/\alpha}}\right)^{-1-\alpha} g_{1/\alpha}\left(\left(-\frac{u}{t^{1/\alpha}}\right)^{-\alpha}\right) g(x-u)\,dt\,du \\
      =& -\int_{-\infty}^x \int_0^\infty  \frac{d}{d\lambda}\left[e^{-\lambda t}\right]\frac{1}{(-u)^{1+\alpha}} g_{1/\alpha}\left(\frac{t}{(-u)^{\alpha}}\right)\,dt\, g(x-u)\,du \\
         =&  \int_{-\infty}^x \frac{1}{u} \frac{d}{d\lambda}\int_0^\infty  e^{-\lambda t}\frac{1}{(-u)^{\alpha}} g_{1/\alpha}\left(\frac{t}{(-u)^{\alpha}}\right)\,dt\, g(x-u)\,du .
   \end{split}
\end{equation*}
Using \eqref{dualityLem1gLT} we have
\begin{equation}\label{dualityLem1eq1}
\int_0^\infty  e^{-\lambda t}\frac{1}{(-u)^{\alpha}} g_{1/\alpha}\left(\frac{t}{(-u)^{\alpha}}\right)\,dt=\exp(-[\lambda(-u)^\alpha]^\beta)=\exp(u\lambda^{1/\alpha}),
\end{equation}
and then the first term in \eqref{resStable} equals
 \begin{equation*}\begin{split}
     \int_{-\infty}^x \frac{1}{u} \frac{d}{d\lambda} \left[e^{u\lambda^{1/\alpha}}\right]\, g(x-u)\,du
     =&\frac{1}{\alpha }\lambda^{1/\alpha-1}\int_0^\infty e^{(x-y)\lambda^{1/\alpha}}g(y)\,dy\\
         =&\frac{1}{\alpha }\lambda^{1/\alpha-1}e^{x\lambda^{1/\alpha}} \mathcal L g(\lambda^{1/\alpha})\\
     =&\frac{\mathcal L g(\lambda^{1/\alpha})}{\alpha\lambda^{1-1/\alpha}}\int_0^\infty e^{-\lambda t}\frac{1}{(-x)^\alpha}g_{1/\alpha}\left(\frac{t}{(-x)^\alpha}\right)\,dt\\
     =&\frac{\mathcal L g(\lambda^{1/\alpha})}{\lambda^{1-1/\alpha}}\int_0^\infty e^{-\lambda t}\frac{(-x)}{\alpha t^{1+1/\alpha}}g_\alpha\left(\frac{x}{t^{1/\alpha}}\right)\,dt
  \end{split}
\end{equation*}
where we have used \eqref{dualityLem1eq1} again the next-to-last line.  Since the last line above is the negative of the second term in \eqref{resStable}, it follows that the sum of these two terms vanishes for all $x<0$.

Since the positively skewed stable density $g_\alpha(x)$ tends to zero at a super-exponential rate as $x\to-\infty$ (e.g., see \cite[Theorem 2.5.2]{zolotarev-61}), its bilateral Laplace transform $\LL_{-\infty}[g_\alpha](s)$ is well defined for all $s>0$, and in fact we can write
\begin{equation}\label{dualityLem1eq2}
\int_{-\infty}^\infty e^{-sx}\frac{1}{t^{1/\alpha}}g_\alpha\left(\frac{x}{t^{1/\alpha}}\right) \,dx=e^{ts^\alpha}\quad\text{for all $t>0$ and all $s>0$.}
\end{equation}
Then it follows that
\begin{equation}\label{dualityLem1eq2a}
\int_0^\infty e^{-\lambda t}\int_{-\infty}^\infty e^{-sx}\frac{1}{t^{1/\alpha}}g_\alpha\left(\frac{x}{t^{1/\alpha}}\right) \,dx\,dt=\frac{1}{\lambda-s^\alpha} \quad\text{for all $s>0$ and all $\Re\lambda>0$.}
\end{equation}
Then for any $g\in C_\infty[0,\infty)$ the convolution property of the bilateral Laplace transform implies that the first term in \eqref{resStable} satisfies
\begin{equation}\label{dualityLem1eq3}
\int_{-\infty}^\infty e^{-sx}\int_{-\infty}^\infty\int_0^\infty e^{-\lambda t}\frac{1}{t^{1/\alpha}}g_\alpha\left(\frac{x-\xi}{t^{1/\alpha}}\right) g(\xi)\,dt\,d\xi\,dx=\frac{\tilde g(s)}{\lambda-s^\alpha}
\end{equation}
for all $s>0$ and all $\Re\lambda>0$.

As for the second term, in view of the fact that the complex contour integral
 \begin{equation*}
     \int_{\lambda}^\infty \frac{d}{ds}\left[\frac{1}{u-s^\alpha} \right]\,du=\frac{\alpha s^{\alpha-1}}{\lambda-s^\alpha}
\end{equation*}
(integrate along $\{\lambda+r:r>0\}$) it follows using \eqref{dualityLem1eq2a} that
 \begin{equation*}\begin{split}
\frac{s^{\alpha-1}}{\lambda-s^\alpha}
&=\frac1\alpha \int_{\lambda}^\infty \frac{d}{ds}\left[\int_0^\infty e^{-u t}\int_{-\infty}^\infty e^{-sx}\frac{1}{t^{1/\alpha}}g_\alpha\left(\frac{x}{t^{1/\alpha}}\right) \,dx\,dt\right]\,du\\
&=\int_{-\infty}^\infty e^{-sx}\int_0^\infty \left(\frac{-x}\alpha\right)\frac{1}{t^{1/\alpha}}g_\alpha\left(\frac{x}{t^{1/\alpha}}\right) \left[ \int_{\lambda}^\infty e^{-u t}\,du\right] \,dt\,dx\\
&=-\int_{-\infty}^\infty e^{-sx}\int_0^\infty e^{-\lambda t}\left(\frac{x}\alpha\right)\frac{1}{t^{1+1/\alpha}}g_\alpha\left(\frac{x}{t^{1/\alpha}}\right) \,dt\,dx .
      \end{split}
\end{equation*}
Then it follows immediately that the bilateral Laplace transform of the second term in \eqref{resStable} equals the second term in \eqref{resL}.  Since the right-hand side of \eqref{resStable} vanishes for $x<0$, its bilateral Laplace transform equals its ordinary Laplace transform.  (Note however that neither term on the right-hand side of \eqref{resStable} vanishes for $x<0$, only their sum.)  But then \eqref{resStable} has the same Laplace transform as $R(\lambda,A)g(x)$, and since both are continuous in view of \eqref{Mittag-Leffler}, the result follows using the uniqueness of the Laplace transform.
\end{proof}


\begin{lemma}\label{LemResolvent}
Under the assumptions of Theorem \ref{bwgenerator}, we have for every $1<\alpha\leq 2$ that
\eqref{resEst} holds
for all $\Re \lambda>0$ and all $g\in C_\infty[0,\infty)$, with $M=M_\alpha+1+\sec(\pi/(2\alpha))$ for some $M_\alpha$ depending only on $\alpha$.
\end{lemma}

\begin{proof}
Denote by $\mathrm{BUC}(\R)$ the Banach space of bounded, uniformly continuous functions with the supremum norm.
Given $g\in C_\infty[0,\infty)$, define $\bar g\in\mathrm{BUC}(\R)$ by setting $\bar g(x)=g(x)$ for $x>0$ and $\bar g(x)=g(0)$ for $x\leq 0$.  Use \eqref{resStable} to write $R(\lambda,A)g(x)=I_1-I_2+I_3$ where
\begin{equation}\begin{split}\label{LemResolventEq1}
I_1&=\int_0^\infty\int_{-\infty}^\infty e^{-\lambda t}\frac{1}{t^{1/\alpha}}g_\alpha\left(\frac{x-\xi}{t^{1/\alpha}}\right) \bar g(\xi)\,d\xi\,dt\\
I_2&=\int_0^\infty\int_{-\infty}^0 e^{-\lambda t}\frac{1}{t^{1/\alpha}}g_\alpha\left(\frac{x-\xi}{t^{1/\alpha}}\right) \bar g(\xi)\,d\xi\,dt\\
I_3&=\frac{\mathcal L g(\lambda^{1/\alpha})}{\lambda^{1-1/\alpha}}\int_0^\infty e^{-\lambda t}\frac{x}{\alpha t^{1+1/\alpha}}g_\alpha\left(\frac{x}{t^{1/\alpha}}\right)\,dt .
  \end{split}
\end{equation}
The formula
\begin{align*}
&T^\alpha_tf(x)=\int_{-\infty}^\infty \frac{1}{t^{1/\alpha}}g_\alpha\left(\frac{x-\xi}{t^{1/\alpha}}\right) f(\xi)\,d\xi\\
&=\int_{-\infty}^\infty \frac{1}{t^{1/\alpha}}g_\alpha\left(\frac{\xi}{t^{1/\alpha}}\right) f(x-\xi)\,d\xi:=\int_{-\infty}^\infty g_{t,\alpha}(\xi)f(x-\xi)\,d\xi,\quad f\in \mathrm{BUC}(\R),
\end{align*}
defines a strongly continuous convolution semigroup on $\mathrm{BUC}(\R)$. Indeed, $\bar{T}^{\alpha}_tf=f*g_{t,\alpha}$, $f\in L_1(\R)$, defines a strongly continuous semigroup on $L_1(\R)$, see, e.g. \cite[Theorem 21.4.3]{HillePhillips}, noting that the Fourier transform of $g_{t,\alpha}$ is $e^{t(ik)^{\alpha}}$. The latter follows from \eqref{dualityLem1eq2} which also holds for $s=ik$, as $g_{\alpha}$ is absolutely integrable. Then, $T^\alpha_t$ is a subordinate semigroup (where the right-translation group on $\mathrm{BUC}(\R)$, which is strongly continuous \cite[Chapter I, Section 4.15]{Engel2000}, is subordinated against $\bar{T}_{\alpha,t}$) which is strongly continuous by \cite[Theorem 4.1]{Baeumer2009a}.

Next we show that $T^{\alpha}_t$ is a bounded analytic semigroup on $\mathrm{BUC}(\R)$, by showing that $\|\frac{d}{dt}T^{\alpha}_tf\|=\|A_{\alpha}T^{\alpha}_tf\|\le Mt^{-1}\|f\|$ for some $M> 0$, see \cite[Theorem 3.7.19]{Arendt2001}. Here, the operator $A_{\alpha}$ denotes the generator of $T^{\alpha}_t$. We have that
$$
\left\|\frac{d}{dt}T^{\alpha}_tf\right\|\le \int_{-\infty}^{\infty}\left|\frac{d}{dt}g_{t,\alpha}(x)\right|dx\,\|f\|.
$$
Recall Carlson's inequality (see, \cite{Car,Edwards})
$$
\int_{-\infty}^{\infty}|h(x)|\,dx \le C \left(\int_{-\infty}^{\infty}\left|{\mathcal F}[h](k)\right|^2dk\right)^{1/4} \left(\int_{-\infty}^{\infty}\left|\frac{d}{dk}{\mathcal F}[h](k)\right|^2dk\right)^{1/4}
$$
where ${\mathcal F}[h](k)=\int e^{-ikx}h(x)\,dx$ denotes the Fourier transform of $h$. Note that  ${\mathcal F}[\frac{d}{dt}g_{t,\alpha}](k)=(ik)^{\alpha}e^{t(ik)^{\alpha}}$. It is easy to check, using the formula for the gamma probability density, that
$$
\int_{-\infty}^\infty\left|(ik)^{\alpha}e^{t(ik)^{\alpha}}\right|^2\,dk\le \int_{-\infty}^{\infty}|k|^{2\alpha}e^{-2tc_{\alpha}|k|^{\alpha}}dk\le C_{\alpha}t^{-2-\frac{1}{\alpha}}
$$
for some $c_{\alpha},C_{\alpha}>0$, and that
$$
\int_{-\infty}^\infty\left|\frac{d}{dk}\left(-ik)^{\alpha}e^{t(-ik)^{\alpha}}\right)\right|^2\,dk\le 2\int_{-\infty}^{\infty}\left(|k|^{2(\alpha-1)}+t^2|k|^{2(2\alpha-1)}\right)e^{-2c_{\alpha}t|k|^{\alpha}}dk\le C_{\alpha}t^{-2+\frac{1}{\alpha}}.
$$
Hence $\|A_{\alpha}T^{\alpha}_tf\|\le K_{\alpha}t^{-1}\|f\|$, and so $T^{\alpha}_t$ is a bounded analytic semigroup on $\mathrm{BUC}(\R)$. Therefore, by Corollary \cite[Corollary 3.7.12]{Arendt2001},
$$
\|\lambda I_1\|=\left\|\lambda \int_{0}^{\infty}e^{-\lambda t}T^{\alpha}_t\bar{g}\,dt\right\|=\|\lambda R(\lambda,A_{\alpha})\bar{g}\|\le M_{\alpha}\|\bar{g}\|=M_{\alpha}\|g\|,\quad \Re \lambda >0.
$$

Next write
\begin{equation}
  \begin{split}
    |\lambda I_2|\leq& |\lambda g(0)|\left\|\int_0^\infty\int_{-\infty}^0 e^{-\lambda t} \frac{1}{t^{1/\alpha}} g_\alpha\left(\frac{x-\xi}{t^{1/\alpha}}\right)\,d\xi \,dt\right\|\\
    \le& |\lambda| \|g\|\left\|\int_0^\infty e^{-\lambda t}\int_{\frac{x}{t^{1/\alpha}}}^\infty g_\alpha(y)\,dy \,dt\right\|\\
    =& |\lambda| \|g\|\left\|\left.-\frac{e^{-\lambda t}}{\lambda}\int_{\frac{x}{t^{1/\alpha}}}^\infty g_\alpha(y)\,dy \right|_0^\infty +\frac1\lambda\int_0^\infty e^{-\lambda t}\frac{x}{\alpha t^{1+1/\alpha}} g_\alpha\left(\frac{x}{t^{1/\alpha}}\right) \,dt\right\|
  \end{split}
\end{equation}
A substitution $t=(x/u)^\alpha$ leads to
\begin{equation}\label{LemResolventEq1x}
\int_0^\infty e^{-\lambda t}\frac{x}{\alpha t^{1+1/\alpha}}g_\alpha\left(\frac{x}{t^{1/\alpha}}\right)\,dt =\int_0^\infty e^{-\lambda (x/u)^\alpha}g_\alpha\left(u\right)\,du ,
\end{equation}
and then it follows that
\[|\lambda I_2|\leq \|g\| \left\|\int_0^\infty e^{-\lambda\left(\frac{x}{y}\right)^\alpha}g_\alpha(y)\,dy\right\|\le \|g\|\]
for all $\Re \lambda>0$ and all $g\in C_\infty[0,\infty)$.

Another application of \eqref{LemResolventEq1x} shows that $|\lambda I_3|\leq |\lambda^{1/\alpha}\mathcal L g(\lambda^{1/\alpha}) |$.  If $\lambda=re^{i\theta}$ for some $r>0$ and $|\theta|<\pi/2$ then $\lambda^{1/\alpha}=r^{1/\alpha}e^{i\theta/\alpha}$ has real part $r^{1/\alpha}\cos(\theta/\alpha)\geq r^{1/\alpha}\cos(\pi/(2\alpha))$.  Hence $|\lambda^{1/\alpha}|=r^{1/\alpha}\leq  \Re[\lambda^{1/\alpha}]/\cos(\pi/(2\alpha))$ for all $\Re \lambda>0$.  It is not hard to check that $\Re[\lambda |\mathcal L g(\lambda)|]\leq \|g\|$ for all $\Re \lambda>0$ and all $g\in C_\infty[0,\infty)$.  Then we have
\[|\lambda I_3|\leq |\lambda^{1/\alpha}\mathcal L g(\lambda^{1/\alpha}) |=\left|\frac{|\lambda^{1/\alpha}|}{\Re[\lambda^{1/\alpha}]}\cdot\Re[\lambda^{1/\alpha}|\mathcal L g(\lambda^{1/\alpha})|]\right|\leq \frac 1{\cos(\pi/(2\alpha))}\|g\| \]
for all $\Re \lambda>0$ and all $g\in C_\infty[0,\infty)$, and the result follows.
\end{proof}

\begin{remark}
Following a slightly different path using Fourier transforms, it is possible to show that
\[\lambda R(\lambda,A)g(x)
=\lambda R(\lambda,D_x^\alpha)\bar g(x) - \lambda R(\lambda,D_x^\alpha) g(0)\I_{(-\infty,0)}(x) + c_\alpha(\lambda R(\lambda,D_x^\alpha)-\mathrm{I})\I_{(-\infty,0)}(x)
\]
where $\bar g(x)=g(x)$ for $x>0$, $\bar g(x)=g(0)$ for $x\leq 0$, $c_\alpha=\lambda^{1/\alpha}\LL g(\lambda^{1/\alpha})$, $\mathrm{I}$ is the identity operator, and $\I_{(-\infty,0)}(x)$ is the indicator function.  This form clarifies that the resolvent of the Caputo fractional derivative $A=\partial_x^\alpha$ is a modification of the resolvent of the Riemann-Liouville fractional derivative $D_x^\alpha$ to account for the boundary term, which is natural in view of \eqref{RLtoCaputoto}.  The same resolvent bound in Lemma \ref{LemResolvent} can also be obtained using this form.
\end{remark}

\section*{Acknowledgments}
The authors would like to thank Zhen-Qing Chen, University of Washington, Pierre Patie, Universit\'e Libre de Bruxelles, and Victor Rivero, Centro de Investigaci\'on en Matem\'aticas for helpful discussions.  We would also like to thank an anonymous referee for many helpful comments that significantly improved the presentation.

\section*{Appendix}
The following Matlab code computes the transition density $p(x,y,t)$ in the case $x=2$ for the reflected stable process $Z_t$ defined by \eqref{ZtDef}, where $Y_t$ is a stable L\'evy process with characteristic function \eqref{YtCF} and index $1<\alpha<2$.  This code was used to generate the plots in Figures~\ref{fig1_2} and~\ref{fig1_8}.
\vskip10pt
\begin{verbatim}
%%% Matlab script to compute p(x,y,t)
%% enter variables
    alpha=1.2; ymax=12; N=1200; t=[0,.5,1,2]; x=2;
%% initialise parameters
    h=ymax/N; y=(h:h:ymax)';
    u0=zeros(N,1);u0(floor(x/h)+1)=1/h; % initial condition
%% Make Grunwald matrix
    w=ones(1,N+1);
    for k=1:N
      w(k+1)=w(k)*(k-alpha-1)/k;
    end
    w=w/h^alpha;
    M=spdiags(repmat(w,N,1),-1:1:N-1,N,N); %enter w's along diagonals
    M(1,:)=-cumsum(w(1:N))'; %change first row for BC
%% Solve ODE system
    [~,p]=ode113(@(t,u) M*u,t,u0);
\end{verbatim}

%

\end{document}